\documentclass[12pt]{article}
\usepackage{graphicx}
\usepackage[ruled,vlined]{algorithm2e}
\usepackage{appendix}
\usepackage{amsmath}
\usepackage{amssymb}
\usepackage{amsthm}
\usepackage{multirow}
\usepackage{color}
\usepackage{longtable}
\usepackage{array}
\usepackage{url}
\usepackage{booktabs}
\usepackage{float}
\usepackage{mathtools}
\usepackage{tikz}
\usepackage{caption}

\newtheorem{theorem}{Theorem}[section]
\newtheorem{definition}[theorem]{Definition}
\newtheorem{lemma}[theorem]{Lemma}

\newtheorem{conjecture}[theorem]{Conjecture}

\theoremstyle{definition}

\title{Fuglede's conjecture holds in $\mathbb{Z}_{p}\times\mathbb{Z}_{p^{n}}$}
\author{Tao Zhang\thanks{Research supported by the National Natural Science Foundation of China under Grant No. 11801109.}\\
\footnotesize  School of Mathematics and Information Science, Guangzhou University, Guangzhou 510006, China. \\
}

\begin{document}

\date{}

\maketitle

\begin{abstract}
Fuglede's conjecture states that for a subset $\Omega$ of a locally compact abelian group $G$ with positive and finite Haar measure, there exists a subset of the dual group of $G$ which is an orthogonal basis of $L^{2}(\Omega)$ if and only if it tiles the group by translation. In this paper, we prove a divisibility property for a set in $\mathbb{Z}_{p}\times\mathbb{Z}_{p^{n}}$. Then using the divisibility property and equi-distributed property, we prove that Fuglede's conjecture holds in the group $\mathbb{Z}_{p}\times\mathbb{Z}_{p^{n}}$.

\medskip

\noindent {{\it Key words and phrases\/}: Fuglede's conjecture, tile, spectral set.}

\smallskip

\noindent {{\it AMS subject classifications\/}: 05B25, 52C20, 11P99, 42B05, 43A40.}
\end{abstract}

\section{Introduction}
Fourier analysis can be viewed as a tool to study an arbitrary function $f$ on (say) the reals $\mathbb{R}$, by looking at how such a function decomposes into a series of orthogonal basis. In this paper, we consider a generalization of this concept.
\begin{definition}
A bounded measurable subset $\Omega\subseteq\mathbb{R}^{n}$ with $\mu(\Omega)>0$ is called spectral, if there is a subset $\Lambda\subseteq\mathbb{R}^{n}$ such that the set of exponential functions $\{e_{\lambda}(x)\}_{\lambda\in\Lambda}$ is a complete orthogonal basis, where $e_{\lambda}(x)=e^{2\pi i\langle\lambda,x\rangle}$, that is
\[\langle e_{\lambda},e_{\lambda'}\rangle=\int_{\Omega}e_{\lambda-\lambda'}(x)dx=\delta_{\lambda,\lambda'}\mu(\Omega),\]
where $\delta_{\lambda,\lambda'}=1$ if $\lambda=\lambda'$ and $0$ otherwise,
and every $f\in L^{2}(\Omega)$ can be expressed as
\[f(x)=\sum_{\lambda\in\Lambda}a_{\lambda}e_{\lambda}(x),\]
for some $a_{\lambda}\in\mathbb{C}$. $\Lambda$ is called the spectrum of $\Omega$, and $(\Omega,\Lambda)$ is called a spectral pair in $\mathbb{R}^{n}$.
\end{definition}
The other class of sets under study is that of tiling sets. Tiling is one of the most diverse and ubiquitous concepts of modern mathematics.
\begin{definition}
A subset $A\subseteq\mathbb{R}^{n}$ tiles $\mathbb{R}^{n}$ by translation, if there is a translation set $T\subseteq\mathbb{R}^{n}$ such that almost all elements of $\mathbb{R}^{n}$ have a unique representation as a sum $a+t$, where $a\in A$, $t\in T$. We will denote this by $A\oplus T=\mathbb{R}^{n}$. $T$ is called the tiling complement of $A$, and $(A,T)$ is called a tiling pair in $\mathbb{R}^{n}$.
\end{definition}
In \cite{F74}, Fuglede proposed the following conjecture, which connected these two notions.
\begin{conjecture}
A subset $\Omega\subseteq\mathbb{R}^{n}$ of positive and finite Lebesgue measure is a spectral set if and only if it tiles $\mathbb{R}^{n}$ by translation.
\end{conjecture}
Fuglede \cite{F74} proved this conjecture under the additional assumption that the tiling set or the spectrum is a lattice subset of $\mathbb{R}^{n}$. After some positive results, Tao \cite{T04} disproved this conjecture by constructing a spectral set in $\mathbb{R}^{5}$ which is not a tile. Later, tiles with no spectrum in $\mathbb{R}^{5}$ had been found by Kolountzakis and Matolcsi \cite{KM06}. Currently, Fuglede's conjecture was disproved in $\mathbb{R}^{n}$ for $n\ge3$ in both directions \cite{AABF17,FMM06,FR06,FS20,K00,KM2006,KM06,M05}. Nonetheless, some important cases had been proven to be true. Iosevich, Katz and Tao \cite{IKT03} proved that Fuglede's conjecture holds for convex sets in $\mathbb{R}^{2}$, then a similar result in dimension 3 was proved by Greenfeld and Lev \cite{GL17}. Recently, Lev and Matolsci \cite{LM21} showed that Fuglede's conjecture holds for convex domains in $\mathbb{R}^{n}$ for all $n$.

It is natural to formulate the following conjecture, which is in a more general setting, the locally compact abelian group.
\begin{conjecture}
Let $G$ be a locally compact abelian group. A subset $\Omega\subseteq G$ of positive and finite Haar measure is a spectral set if and only if it is a tile.
\end{conjecture}
In this general setting, the conjecture is far from being solved and is false for some specific groups as Tao \cite{T04} showed. Hence, the question becomes for which abelian group $G$, Fuglede's conjecture holds. In \cite{FFLS19}, Fan et al. proved that Fuglede's conjecture holds in $\mathbb{Q}_{p}$, the field of $p$-adic numbers. In the following of this paper, we focus on finite abelian groups. Let $G$ be a finite abelian group, by the Fundamental Theorem of Finite Abelian Groups, $G\cong \mathbb{Z}_{m_{1}}\times\mathbb{Z}_{m_{2}}\times\dots\times\mathbb{Z}_{m_{t}}$, where $t$ is the number of generators, $\mathbb{Z}_{m_{i}}$ ($i=1,\dots,t$) is a cyclic group of order $m_{i}$, and $m_{1}\mid m_{2}\mid\dots\mid m_{t}$, $m_{1}>1$. If the abelian group with at least 3 generators, then there exists a finite abelian group such that the Fuglede's conjecture fails for both directions \cite{AABF17,FMM06,FR06,K00,KM2006,KM06,M05}. For abelian groups with 2 generators, Fuglede's conjecture holds in $\mathbb{Z}_{p}^{2}$ \cite{IMP17}, $\mathbb{Z}_{p}\times\mathbb{Z}_{p^{2}}$ \cite{S20}, $\mathbb{Z}_{p}\times\mathbb{Z}_{pq}$ \cite{KS2021}  and $\mathbb{Z}_{pq}\times\mathbb{Z}_{pq}$ \cite{FKS2012}, where $p,q$ are distinct primes. For abelian groups with 1 generator, Fuglede's conjecture holds in cyclic group $\mathbb{Z}_{N}$, where $N=p^{n},p^{n}q,p^{n}q^{2},pqr,p^{2}qr,pqrs$ \cite{FFS16,KMSV2012,KMSV20,L02,MK17,M21,S19,Somlai21}, and $p,q,r,s$ are distinct primes.

The primary aim of this paper is to continue this investigation and our main result is the following.
\begin{theorem}\label{mainthm}
Let $p$ be a prime and $n$ be a positive integer. A subset in $\mathbb{Z}_{p}\times\mathbb{Z}_{p^{n}}$ is a spectral set if and only if it is a tile of $\mathbb{Z}_{p}\times\mathbb{Z}_{p^{n}}$.
\end{theorem}

When we consider the Fuglede's conjecture in cyclic group $\mathbb{Z}_{N}$, one of the most important tools is the so called (T1) and (T2) conditions, which was introduced by Coven and Meyerowitz \cite{CM99}. In this paper, we will show a divisibility property in $\mathbb{Z}_{p}\times\mathbb{Z}_{p^{n}}$ in Lemma~\ref{lemma2}, which is a weaker property than (T1) condition. However, it is very useful for bounding the size of a set. Another useful tool to investigate the Fuglede's conjecture in $\mathbb{Z}_{p}\times\mathbb{Z}_{p^{n}}$ is the equi-distributed property, which has been proven in \cite{S20}. The divisibility property and equi-distributed property are the main tools to prove Theorem~\ref{mainthm}.

This paper is organized as follows. In Section~\ref{sectionproperty}, we give some basic results about spectral sets and tiles in finite abelian groups. In Section~\ref{sectionsetin}, we show some properties of sets in $\mathbb{Z}_{p}\times\mathbb{Z}_{p^{n}}$. In Section~\ref{sectionts}, we prove the Tile$\Rightarrow$Spectral direction of Theorem~\ref{mainthm}, and in Section~\ref{sectionst}, we prove the Spectral$\Rightarrow$Tile direction of Theorem~\ref{mainthm}. Section~\ref{conclusion} concludes the paper.
\section{Properties of spectral sets and tiles}\label{sectionproperty}
In this section, we give some properties of spectral sets and tiles in finite abelian groups. The following notations are fixed throughout this paper.
\begin{itemize}
  \item Let $a,b$ be integers such that $a\le b$, denote $[a,b]=\{a,a+1,\dots,b\}$.
  \item For any positive integer $m$, $\mathbb{Z}_{m}$ is the ring of integers modulo $m$, and $\mathbb{Z}_{m}^{*}=\{a\in\mathbb{Z}_{m}:\ \gcd(a,m)=1\}$.
  \item Let $p$ be a prime, for any $t\in[0,p^{m}-1]$, let $t_{p}[i]$ be the $i$-th coefficient of $p$-adic representation of $t$, i.e., $t=\sum_{i=0}^{m-1}t_{p}[i]p^{i}$, where $t_{p}[i]\in[0,p-1]$. We will simply use $t[i]$ instead of $t_{p}[i]$ when there is no misunderstanding. Let $M(t)$ be the minimal $i$ such that $t[i]\ne0$, i.e.,
      \[M(t):=\min\{i\in[0,m-1]: t[i]\ne0\}.\]
\end{itemize}
 Let $G$ be a finite abelian group, then $G\cong \mathbb{Z}_{m_{1}}\times\mathbb{Z}_{m_{2}}\times\dots\times\mathbb{Z}_{m_{t}}$, for some $m_{i}$ with $m_{1}\mid m_{2}\mid\dots\mid m_{t}$ and $m_{1}>1$. For any $g\in G$, we will represent $g$ as $(g_{1},\dots,g_{t})$, where $g_{i}\in\mathbb{Z}_{m_{i}}$.
For any $g,h\in G$, we define their inner product by
 \[\langle g,h\rangle:=\sum_{i=1}^{t}g_{i}h_{i}\frac{m_{t}}{m_{i}}.\]
 Let
\[\chi_{g}(h)=e^{\frac{2\pi i\cdot \langle g,h\rangle}{m_{t}}},\]
and $\chi_{g}\circ\chi_{h}=\chi_{g+h}$.
Then the set $\widehat{G}=\{\chi_{g}:\ g\in G\}$ with $\circ$ forms a group which is isomorphic to $G$.

Let $\mathbb{Z}[G]$ denote the group ring of $G$ over $\mathbb{Z}$. For any $A\in\mathbb{Z}[G]$,
$A$ can be written as $A=\sum_{g\in G}a_{g}g$, where $a_{g}\in\mathbb{Z}$. For any $\chi \in \widehat{G}$, define
\[\chi(A)=\sum_{g\in G} a_{g} \chi(g).\]
The following inversion formula shows that $A$ is completely determined
by its character value $\chi(A)$, where $\chi$ ranges over $\widehat{G}$.

\begin{lemma}[(Fourier inversion formula)]\label{inversion formula}
Let $G$ be an abelian group. If $A=\sum_{g\in G}a_g
g\in \mathbb{Z}[G]$, then
\[
a_g=\frac{1}{|G|}\sum_{\chi\in\widehat G}\chi(A) \chi (g^{-1}),
\]
for all $g\in G$.
\end{lemma}

Since we mainly work on the finite abelian group, we restate the definition of spectral set.
\begin{definition}
Let $G$ be a finite abelian group.
A subset $A\subseteq G$ is said to be spectral if there is a subset $B\subseteq G$ such that
\[\{\chi_{b}: b\in B\}\]
forms an orthogonal basis in $L^{2}(A)$, the vector space of complex valued functions on $A$ with Hermitian inner product $\langle f,g\rangle=\sum_{a\in A}f(a)\overline{g}(a)$. In such case, the set $B$ is called a spectrum of $A$, and $(A,B)$ is called a spectral pair.
\end{definition}
Since the dimension of $L^{2}(A)$ is $|A|$, the pair $(A,B)$ being a spectral pair is equivalent to
\[|A|=|B|\textup{ and }\chi_{b-b'}(A)=0\textup{ for all }b\ne b'\in B.\]
The set of zeros of $A$ is defined by
\[\mathcal{Z}_{A}=\{g\in G: \chi_{g}(A)=0\}.\]

\begin{lemma}\label{prelemma3}
If $g\in\mathcal{Z}_{A}$, then $rg\in\mathcal{Z}_{A}$ for all $r\in\mathbb{Z}_{m_{t}}^{*}$.
\end{lemma}
\begin{proof}
Since $g\in\mathcal{Z}_{A}$, then we have
\[\chi_{g}(A)=\sum_{h\in A}e^{\frac{2\pi i\cdot \langle g,h\rangle}{m_{t}}}=0.\]
By Galois theory,
\[\chi_{rg}(A)=\sum_{h\in A}e^{\frac{2\pi i\cdot \langle rg,h\rangle}{m_{t}}}=\sum_{h\in A}e^{\frac{2\pi i\cdot r\langle g,h\rangle}{m_{t}}}=0\]
for all $r\in\mathbb{Z}_{m_{t}}^{*}$, which means that $rg\in\mathcal{Z}_{A}$.
\end{proof}

Now we give some equivalent conditions of spectral pair.
\begin{lemma}\label{lemma3}
Let $A,B\subseteq G$. Then the following statements are equivalent.
\begin{enumerate}
  \item[(a)] $(A,B)$ is a spectral pair.
  \item[(b)] $(B,A)$ is a spectral pair.
  \item[(c)] $|A|=|B|$ and $(B-B)\backslash\{0\}\subseteq\mathcal{Z}_{A}$.
  \item[(d)] The pair $(aA+g,bB+h)$ is a spectral pair for all $a,b\in\mathbb{Z}_{m_{t}}^{*}$ and $g,h\in G$.
\end{enumerate}
\end{lemma}
\begin{proof}
The equivalence of (a), (b) and (c) can be found in \cite{S20}.

For (a)$\Rightarrow$(d), note that $|B|=|bB+h|$ and $(bB+h)-(bB+h)=b(B-B)$, by (c) and Lemma~\ref{prelemma3}, $(A,bB+h)$ is a spectral pair. By (b), $(bB+h,A)$ is a spectral pair, then a similar discussion as above, $(aA+g,bB+h)$ is a spectral pair. For (d)$\Rightarrow$(a), we only need to take $a=b=1$ and $g=h=0$.
\end{proof}

\begin{definition}
Let $G$ be a finite abelian group.
A subset $A\subseteq G$ is said to be a tile if there is a subset $T\subseteq G$ such that each element $g\in G$ can be expressed uniquely in the form
\[g=a+t,\ a\in A,\ t\in T.\]
We will denote this by $G=A\oplus T$.
The set $T$ is called a tiling complement of $A$, and $(A,T)$ is called a tiling pair.
\end{definition}
The following result can be found in \cite{S20}, \cite[Lemma 2.1]{SS09}.
\begin{lemma}\label{prelemma2}
Let $A,T$ be subsets in $G$. Then the following statements are equivalent.
\begin{enumerate}
  \item[(a)] $(A,T)$ is a tiling pair.
  \item[(b)] $(T,A)$ is a tiling pair.
  \item[(c)] $(A+g,T+h)$ is a tiling pair.
  \item[(d)] $|A|\cdot|T|=|G|$ and $(A-A)\cap(T-T)=\{0\}$.
  \item[(e)] $|A|\cdot|T|=|G|$ and $\mathcal{Z}_{A}\cup\mathcal{Z}_{T}=G\backslash\{0\}$.
\end{enumerate}
\end{lemma}

If $|A|=1$ or $A=G$, then the set $A$ is called trivial. It is easy to see that a trivial set is a spectral set and also a tiling set. In the following of this paper, we will only consider nontrivial set.
\begin{lemma}\label{lemma4}
If $(A,T)$ is a nontrivial tiling pair of $G$, then $\mathcal{Z}_{A}\ne\emptyset$ and $\mathcal{Z}_{T}\ne\emptyset$.
\end{lemma}
\begin{proof}
If $\mathcal{Z}_{A}=\emptyset$, then by Lemma~\ref{prelemma2}, $\mathcal{Z}_{T}=G\backslash\{0\}$. By the Fourier inversion formula (Lemma~\ref{inversion formula}), we have $T=G$, then $|A|=1$, which is a contradiction. Hence $\mathcal{Z}_{A}\ne\emptyset$. Similarly, we can get $\mathcal{Z}_{T}\ne\emptyset$.
\end{proof}

\section{Sets in $\mathbb{Z}_{p}\times\mathbb{Z}_{p^{n}}$}\label{sectionsetin}
Recall that for $u,v\in\mathbb{Z}_{p}\times\mathbb{Z}_{p^{n}}$, their inner product is
\[\langle u,v\rangle=p^{n-1}u_{1}v_{1}+u_{2}v_{2}\in\mathbb{Z}_{p^{n}}.\]
Define
\[H(u,t):=\{x\in\mathbb{Z}_{p}\times\mathbb{Z}_{p^{n}}: \langle x,u\rangle=t\},\]
and
\[H_{A}(u,t):=H(u,t)\cap A,\]
for $u\in\mathbb{Z}_{p}\times\mathbb{Z}_{p^{n}}$, $t\in\mathbb{Z}_{p^{n}}$ and $A\subseteq \mathbb{Z}_{p}\times\mathbb{Z}_{p^{n}}$. Then the following lemma, which is called equi-distributed property, can be found in \cite{S20}.
\begin{lemma}\label{lemma1}
Let $A\subseteq\mathbb{Z}_{p}\times\mathbb{Z}_{p^{n}}$ and $u\in\mathbb{Z}_{p}\times\mathbb{Z}_{p^{n}}$. Then $u\in\mathcal{Z}_{A}$ if and only if $|H_{A}(u,t)|=|H_{A}(u,t')|$ for all $t'\equiv t\pmod{p^{n-1}}$.
\end{lemma}

Now we prove the divisibility property for a set in $\mathbb{Z}_{p}\times\mathbb{Z}_{p^{n}}$, which is useful in the following sections.
\begin{lemma}\label{lemma2}
Let $A\subseteq\mathbb{Z}_{p}\times\mathbb{Z}_{p^{n}}$. If $(a,p^{i_{1}}),(0,p^{i_{2}}),\dots,(0,p^{i_{s}})\in\mathcal{Z}_{A}$ for some $a\in\mathbb{Z}_{p}$ and $0\le i_{1}<i_{2}<\dots<i_{s}\le n-1$, then $p^{s}\mid|A|$.
\end{lemma}
\begin{proof}
For any $2\le r\le s$, $v\in[0,p^{n-1-i_{r}}-1]$ and $k\in[0,p-1]$, define
\[S(r,v,k)=\{u\in[0,p^{n-1-i_{r-1}}-1]: u[i]=v[i]\textup{ for }i\in[0,n-2-i_{r}]\textup{ and }u[n-1-i_{r}]=k\}.\]

{\bf{Claim:}} For any $a\in\mathbb{Z}_{p}$, $v\in[0,p^{n-1-i_{r}}-1]$ and $k\in[0,p-1]$, we have
\[\cup_{u\in S(r,v,k)}\cup_{j=0}^{p-1}H_{A}((a,p^{i_{r-1}}),up^{i_{r-1}}+jp^{n-1})=H_{A}((0,p^{i_{r}}),vp^{i_{r}}+kp^{n-1}).\]

If $(x,y)\in\cup_{u\in S(r,v,k)}\cup_{j=0}^{p-1}H_{A}((a,p^{i_{r-1}}),up^{i_{r-1}}+jp^{n-1})$, then we have
\[axp^{n-1}+yp^{i_{r-1}}\equiv up^{i_{r-1}}+jp^{n-1}\pmod{p^{n}},\]
which leads to $y\equiv u\pmod{p^{n-1-i_{r-1}}}.$ By the definition of $S(r,v,k)$,
\[yp^{i_{r}}\equiv vp^{i_{r}}+kp^{n-1}\pmod{p^{n}},\]
then $(x,y)\in H_{A}((0,p^{i_{r}}),vp^{i_{r}}+kp^{n-1})$. Hence,
\begin{align}
\cup_{u\in S(r,v,k)}\cup_{j=0}^{p-1}H_{A}((a,p^{i_{r-1}}),up^{i_{r-1}}+jp^{n-1})\subseteq H_{A}((0,p^{i_{r}}),vp^{i_{r}}+kp^{n-1}).\label{eq1}
\end{align}
If $(x,y)\in H_{A}((0,p^{i_{r}}),vp^{i_{r}}+kp^{n-1})$, then we have
\[yp^{i_{r}}\equiv vp^{i_{r}}+kp^{n-1}\pmod{p^{n}},\]
which leads to $y\equiv v+kp^{n-1-i_{r}}\pmod{p^{n-i_{r}}}$. Then
\[axp^{n-1}+yp^{i_{r-1}}\equiv up^{i_{r-1}}+jp^{n-1}\pmod{p^{n}}\]
for some $u\in S(r,v,k)$ and $j\in[0,p-1]$. Hence $(x,y)\in \cup_{u\in S(r,v,k)}\cup_{j=0}^{p-1}H_{A}((a,p^{i_{r-1}}),up^{i_{r-1}}+jp^{n-1})$. Therefore
\begin{align}\cup_{u\in S(r,v,k)}\cup_{j=0}^{p-1}H_{A}((a,p^{i_{r-1}}),up^{i_{r-1}}+jp^{n-1})\supseteq H_{A}((0,p^{i_{r}}),vp^{i_{r}}+kp^{n-1}).\label{eq2}
\end{align}
Combining (\ref{eq1}) and (\ref{eq2}), we have
\[\cup_{u\in S(r,v,k)}\cup_{j=0}^{p-1}H_{A}((a,p^{i_{r-1}}),up^{i_{r-1}}+jp^{n-1})=H_{A}((0,p^{i_{r}}),vp^{i_{r}}+kp^{n-1}).\]
This ends the proof of the claim.

Since $(a,p^{i_{1}}),(0,p^{i_{2}}),\dots,(0,p^{i_{s}})\in\mathcal{Z}_{A}$, by Lemma~\ref{lemma1}, we have
\begin{align*}
&|H_{A}((a,p^{i_{1}}),up^{i_{1}}+jp^{n-1})|=|H_{A}((a,p^{i_{1}}),up^{i_{1}}+j'p^{n-1})|,\\
&|H_{A}((0,p^{i_{r}}),vp^{i_{r}}+jp^{n-1})|=|H_{A}((0,p^{i_{r}}),vp^{i_{r}}+j'p^{n-1})|
\end{align*}
for all $u\in[0,p^{n-1-i_{1}}-1]$, $2\le r\le s$, $v\in[0,p^{n-1-i_{r}}-1]$ and $0\le j\ne j'\le p-1$. For any $(x,y)\in A$, we have $yp^{i_{s}}\equiv v_{s}p^{i_{s}}+jp^{n-1}$ for some $v_{s}\in[0,p^{n-1-i_{s}}-1],$ $j\in[0,p-1]$. Hence
\[A=\cup_{j=0}^{p-1}\cup_{v_{s}\in[0,p^{n-1-i_{s}}-1]} H_{A}((0,p^{i_{s}}),v_{s}p^{i_{s}}+jp^{n-1}).\]
Then we have
\begin{align*}
|A|=&|\cup_{j=0}^{p-1}\cup_{v_{s}\in[0,p^{n-1-i_{s}}-1]} H_{A}((0,p^{i_{s}}),v_{s}p^{i_{s}}+jp^{n-1})|\\
=&p|\cup_{v_{s}\in[0,p^{n-1-i_{s}}-1]} H_{A}((0,p^{i_{s}}),v_{s}p^{i_{s}})|\\
=&p|\cup_{j=0}^{p-1}\cup_{v_{s-1}\in S(s,v_{s},0)}\cup_{v_{s}\in[0,p^{n-1-i_{s}}-1]} H_{A}((0,p^{i_{s-1}}),v_{s-1}p^{i_{s-1}}+jp^{n-1})|\\
=&p^{2}|\cup_{v_{s-1}\in S(s,v_{s},0)}\cup_{v_{s}\in[0,p^{n-1-i_{s}}-1]} H_{A}((0,p^{i_{s-1}}),v_{s-1}p^{i_{s-1}})|\\
 =&\dots\\
 =&p^{s}|\cup_{v_{1}\in S(2,v_{2},0)}\dots\cup_{v_{s-1}\in S(s,v_{s},0)}\cup_{v_{s}\in[0,p^{n-1-i_{s}}-1]}H_{A}((a,p^{i_{1}}),v_{1}p^{i_{1}})|.
\end{align*}
Hence $p^{s}\mid |A|.$
\end{proof}

For $u,v\in\mathbb{Z}_{p}\times\mathbb{Z}_{p^{n}}$,
we define the relation $u\sim v$ if there exists $r\in\mathbb{Z}_{p^{n}}^{*}$ such that $u=rv$. Then the equivalent classes in $\mathbb{Z}_{p}\times\mathbb{Z}_{p^{n}}$ by $\sim$ are
\[(1,0),(c,p^{i})\text{ for all $c\in\mathbb{Z}_{p}$ and $i\in[0,n-1]$.}\]
By Lemma~\ref{prelemma3}, when we study the zero set $\mathcal{Z}_{A}$, we only need to consider the above equivalent classes.

\section{Tile$\Rightarrow$Spectral}\label{sectionts}
In this section, we will prove that if $A$ tiles $\mathbb{Z}_{p}\times\mathbb{Z}_{p^{n}}$, then $A$ is a spectral set in $\mathbb{Z}_{p}\times\mathbb{Z}_{p^{n}}$. By Lemma~\ref{prelemma2}, if $A$ is a nontrivial tiling set in $\mathbb{Z}_{p}\times\mathbb{Z}_{p^{n}}$, then $|A|\mid p^{n+1}$, hence $|A|=p^{i}$ for some $1\le i\le n$.
\begin{theorem}
Let $A\subseteq\mathbb{Z}_{p}\times\mathbb{Z}_{p^{n}}$ be a tiling set. If $|A|=p$, then $A$ is a spectral set.
\end{theorem}
\begin{proof}
By Lemma~\ref{lemma4}, $\mathcal{Z}_{A}\ne\emptyset$. Let $(x,y)\in\mathcal{Z}_{A}$, define the set
\[B=\{r(x,y):\ r\in[0,p-1]\}.\]
Note that $(B-B)\backslash\{(0,0)\}\subseteq\mathcal{Z}_{A}$ and $|A|=|B|=p$, by Lemma~\ref{lemma3}, $(A,B)$ is a spectral pair.
\end{proof}

\begin{theorem}
Let $A\subseteq\mathbb{Z}_{p}\times\mathbb{Z}_{p^{n}}$ be a tiling set. If $|A|=p^{t}$, $2\le t\le n$, then $A$ is a spectral set.
\end{theorem}
\begin{proof}
Suppose $(A,T)$ forms a tiling pair in $\mathbb{Z}_{p}\times\mathbb{Z}_{p^{n}}$, then $|T|=p^{n-t+1}$.
Let $I=\{i\in[0,n-1]:\ (0,p^{i})\in\mathcal{Z}_{A}\}$ and $J=\{i\in[0,n-1]:\ (0,p^{i})\in\mathcal{Z}_{T}\}$.
 By Lemma~\ref{lemma2}, we have $|I|\le t$ and $|J|\le n-t+1$.
 On the other hand, $\mathcal{Z}_{A}\cup\mathcal{Z}_{T}=\mathbb{Z}_{p}\times\mathbb{Z}_{p^{n}}\backslash\{(0,0)\}$, then $|I|+|J|\ge n$. Hence $t-1\le|I|\le t$.

If $|I|=t$, define
\[B=\{\sum_{i\in I}s_{i}(0,p^{i}): s_{i}\in[0,p-1]\}.\]
For any $b\ne b'\in B$, $b-b'=\sum_{i\in I}r_{i}(0,p^{i})$, where $r_{i}\in[-p+1,p-1]$.
Note that $\sum_{i\in I}r_{i}(0,p^{i})\sim(0,p^{j})$, where $j=\min\{i\in I: r_{i}\ne0\}$, then $(B-B)\backslash\{(0,0)\}\subseteq\mathcal{Z}_{A}$. Since $|B|=p^{t}=|A|$, then we have $(A,B)$ is a spectral pair.

If $|I|=t-1$, then $|J|=n-t+1$. Suppose $I=\{a_{1},a_{2},\dots,a_{t-1}\}$ and $J=\{b_{1},b_{2},\dots,b_{n-t+1}\}$, where $0\le a_{1}<\dots<a_{t-1}\le n-1$, $0\le b_{1}<\dots<b_{n-t+1}\le n-1$ and $a_{i}\ne b_{j}$. Now we divide our discussion into three cases.

 {\bf{Case 1: There exist $d\in\mathbb{Z}_{p}$ and $a_{k}$ such that $(d,p^{a_{k}})\in\mathcal{Z}_{T}$ and $(c,p^{b_{j}})\in\mathcal{Z}_{T}$ for all $c\in\mathbb{Z}_{p}$, $b_{j}<a_{k}$.}}

Define
\[B=\{s_{0}(d,p^{a_{k}})+\sum_{i=1}^{n-t+1}s_{i}(0,p^{b_{i}}): s_{i}\in[0,p-1]\}.\]
For any $b\ne b'\in B$, $b-b'=r_{0}(d,p^{a_{k}})+\sum_{i=1}^{n-t+1}r_{i}(0,p^{b_{i}})$, where $r_{i}\in[-p+1,p-1]$.
Note that
 \[r_{0}(d,p^{a_{k}})+\sum_{i=1}^{n-t+1}r_{i}(0,p^{b_{i}})\sim\begin{cases}(d,p^{a_{k}}),&\textup{ if } r_{0}\ne0\text{ and } r_{i}=0\text{ for all }b_{i}<a_{k};\\
(r_{l}^{-1}r_{0}d,p^{b_{l}}),&\textup{ if }r_{0}\ne0,\ b_{l}<a_{k},\ r_{l}\ne0\text{ and }r_{i}=0\text{ for all }1\le i<l;\\
(0,p^{b_{l}}),&\textup{ if }r_{0}=0,\ r_{l}\ne0\text{ and }r_{i}=0\text{ for all }1\le i<l.\end{cases}\]
Then $(B-B)\backslash\{(0,0)\}\subseteq\mathcal{Z}_{T}$,
hence $B$ is a spectrum for $T$, but $|B|=p^{n-t+2}>|T|$, which is a contradiction.

 {\bf{Case 2: There exist $d\in\mathbb{Z}_{p}$ and $b_{k}$ such that $(d,p^{b_{k}})\in\mathcal{Z}_{A}$ and $(c,p^{a_{j}})\in\mathcal{Z}_{A}$ for all $c\in\mathbb{Z}_{p}$, $a_{j}<b_{k}$.}}

Define
\[B=\{s_{0}(d,p^{b_{k}})+\sum_{i=1}^{t-1}s_{i}(0,p^{a_{i}}): s_{i}\in[0,p-1]\}.\]
For any $b\ne b'\in B$, $b-b'=r_{0}(d,p^{b_{k}})+\sum_{i=1}^{t-1}r_{i}(0,p^{a_{i}})$, where $r_{i}\in[-p+1,p-1]$. Note that
 \[r_{0}(d,p^{b_{k}})+\sum_{i=1}^{t-1}r_{i}(0,p^{a_{i}})\sim\begin{cases}(d,p^{b_{k}}),&\textup{ if } r_{0}\ne0\text{ and } r_{i}=0\text{ for all }a_{i}<b_{k};\\
(r_{l}^{-1}r_{0}d,p^{a_{l}}),&\textup{ if }r_{0}\ne0,\ a_{l}<b_{k},\ r_{l}\ne0\text{ and }r_{i}=0\text{ for all }1\le i<l;\\
(0,p^{a_{l}}),&\textup{ if }r_{0}=0,\ r_{l}\ne0\text{ and }r_{i}=0\text{ for all }i<l.\end{cases}\]
Then $(B-B)\backslash\{(0,0)\}\subseteq\mathcal{Z}_{A}$. Since $|B|=p^{t}=|A|$, then we have $(A,B)$ is a spectral pair.

{\bf{Case 3: $\{(c,p^{a_{i}}): c\in\mathbb{Z}_{p},i\in[1,t]\}\subseteq\mathcal{Z}_{A}$, and $\{(c,p^{b_{i}}): c\in\mathbb{Z}_{p},i\in[1,n-t+1]\}\subseteq\mathcal{Z}_{T}$.}}

 If $(1,0)\in\mathcal{Z}_{T}$, define
\[B=\{s_{0}(1,0)+\sum_{i=1}^{n-t+1}s_{i}(0,p^{b_{i}}): s_{i}\in[0,p-1]\}.\]
For any $b\ne b'\in B$, $b-b'=r_{0}(1,0)+\sum_{i=1}^{n-t+1}r_{i}(0,p^{b_{i}})$, where $r_{i}\in[-p+1,p-1]$. Note that
 \[r_{0}(1,0)+\sum_{i=1}^{n-t+1}r_{i}(0,p^{b_{i}})\sim\begin{cases}(1,0),&\textup{ if }r_{0}\ne0\text{ and } r_{i}=0\text{ for all }i\ge1;\\
(r_{l}^{-1}r_{0},p^{b_{l}}),&\textup{ if }r_{0}\ne0,\ r_{l}\ne0\text{ and }r_{i}=0\text{ for all }1\le i<l;\\
(0,p^{b_{l}}),&\textup{ if }r_{0}=0,\ r_{l}\ne0\text{ and }r_{i}=0\text{ for all }1\le i<l.\end{cases}\]
Then $(B-B)\backslash\{(0,0)\}\subseteq\mathcal{Z}_{T}$,
 $B$ is a spectrum for $T$, but $|B|=p^{n-t+2}>|T|$, which is a contradiction. Hence $(1,0)\in\mathcal{Z}_{A}$, define
\[C=\{s_{0}(1,0)+\sum_{i=1}^{t-1}s_{i}(0,p^{a_{i}}): s_{i}\in[0,p-1]\}.\]
For any $c\ne c'\in C$, $c-c'=r_{0}(1,0)+\sum_{i=1}^{t-1}r_{i}(0,p^{a_{i}})$, where $r_{i}\in[-p+1,p-1]$. Note that
 \[r_{0}(1,0)+\sum_{i=1}^{t-1}r_{i}(0,p^{a_{i}})\sim\begin{cases}(1,0),&\textup{ if }r_{0}\ne0\text{ and } r_{i}=0\text{ for all }i\ge1;\\
(r_{l}^{-1}r_{0},p^{a_{l}}),&\textup{ if }r_{0}\ne0,\ r_{l}\ne0\text{ and }r_{i}=0\text{ for all }1\le i<l;\\
(0,p^{a_{l}}),&\textup{ if }r_{0}=0,\ r_{l}\ne0\text{ and }r_{i}=0\text{ for all }1\le i<l.\end{cases}\]
Then $(C-C)\backslash\{(0,0)\}\subseteq\mathcal{Z}_{A}$. Since $|C|=p^{t}=|A|$, we have $(A,C)$ is a spectral pair.

We claim that the tiling pair $(A,T)$ satisfies at least one of Case 1, Case 2 and Case 3.
Assume that Case 1 and Case 2 fail. By Lemmas~\ref{prelemma2} and \ref{lemma2}, $\{(c,p^{a_{i}}): c\in[0,p-1],a_{i}<b_{1}\}\subseteq\mathcal{Z}_{A}$. Since Case 2 fails, then $\{(c,p^{b_{1}}): c\in[0,p-1]\}\subseteq\mathcal{Z}_{T}$. Repeating above argument, we get that Case 3 holds. This completes the proof.
\end{proof}
\section{Spectral$\Rightarrow$Tile}\label{sectionst}
In this section, we will prove that if $A$ is a spectral set in $\mathbb{Z}_{p}\times\mathbb{Z}_{p^{n}}$, then $A$ tiles $\mathbb{Z}_{p}\times\mathbb{Z}_{p^{n}}$.
Note that $\mathcal{Z}_{A}\ne\emptyset$, then $p\mid|A|$.
\begin{theorem}
Let $A\subseteq\mathbb{Z}_{p}\times\mathbb{Z}_{p^{n}}$ be a spectral set. If $|A|>p^{n}$, then $A=\mathbb{Z}_{p}\times\mathbb{Z}_{p^{n}}$.
\end{theorem}
\begin{proof}
Suppose $(A,B)$ is a spectral pair in $\mathbb{Z}_{p}\times\mathbb{Z}_{p^{n}}$, then $|A|=|B|>p^{n}$. By pigeonhole principle, there exist $(b,x),(b',x)\in B$ with $b\ne b'$. Then
\[(b,x)-(b',x)\sim(1,0)\in\mathcal{Z}_{A}.\]
For any $i_{0}\in[0, n-1]$ and $c_{0}\in\mathbb{Z}_{p}$, $x\in\mathbb{Z}_{p}\times\mathbb{Z}_{p^{n}}$, $x$ can be represented as
\[x=\sum_{i\in[0,n-1]\backslash\{i_{0}\}}x_{i}(0,p^{i})+x_{n}(1,0)+x_{i_{0}}(c_{0},p^{i_{0}}),\]
where $x_{i}\in[0,p-1]$.
Since $|B|>p^{n}$, by pigeonhole principle, there exist $u,v\in B$ such that
\begin{align*}
&u=\sum_{i\in[0,n-1]\backslash\{i_{0}\}}u_{i}(0,p^{i})+u_{n}(1,0)+u_{i_{0}}(c_{0},p^{i_{0}}),\\
&v=\sum_{i\in[0,n-1]\backslash\{i_{0}\}}u_{i}(0,p^{i})+u_{n}(1,0)+v_{i_{0}}(c_{0},p^{i_{0}}),
\end{align*}
with $u_{i_0}\ne v_{i_{0}}$, then we obtain
\[u-v\sim(c_{0},p^{i_{0}})\in\mathcal{Z}_{A}.\]
Thus, $\mathcal{Z}_{A}=\mathbb{Z}_{p}\times\mathbb{Z}_{p^{n}}\backslash\{(0,0)\}$. By the Fourier inversion formula, $A=\mathbb{Z}_{p}\times\mathbb{Z}_{p^{n}}$.
\end{proof}

It is easy to see that the set $A=\mathbb{Z}_{p}\times\mathbb{Z}_{p^{n}}$ is a spectral set. Now we consider the spectral set $A\subseteq\mathbb{Z}_{p}\times\mathbb{Z}_{p^{n}}$ with $|A|\le p^{n}$.

\begin{theorem}
Let $A\subseteq\mathbb{Z}_{p}\times\mathbb{Z}_{p^{n}}$ be a spectral set. If $|A|=p$, then $A$ is a tiling set.
\end{theorem}
\begin{proof}
Note that $\mathcal{Z}_{A}\ne\emptyset$, we divide our discussion into three cases.

{\bf{Case 1: $(1,0)\in\mathcal{Z}_{A}$}.}

Define
\[T=\{(0,x):\ x\in\mathbb{Z}_{p^{n}}\}.\]
For any $i\in[0,n-1]$, $c\in\mathbb{Z}_{p}$, $t\in[0,p^{n-1-i}-1]$ and $j\in[0,p-1]$, if $(0,x)\in H_{T}((c,p^{i}),tp^{i}+jp^{n-1})$, then $xp^{i}\equiv tp^{i}+jp^{n-1}$, which leads to $x\equiv t+jp^{n-1-i}\pmod{p^{n-i}}$. Hence
\[|H_{T}((c,p^{i}),tp^{i}+jp^{n-1})|=p^{i}\]
for any $i\in[0,n-1]$, $c\in\mathbb{Z}_{p}$, $t\in[0,p^{n-1-i}-1]$ and $j\in[0,p-1].$
By Lemma~\ref{lemma1},
\[(c,p^{i})\in\mathcal{Z}_{T}\text{ for all }i\in[0,n-1],\ c\in\mathbb{Z}_{p}.\]
Then we have
\[\mathcal{Z}_{A}\cup\mathcal{Z}_{T}=(\mathbb{Z}_{p}\times\mathbb{Z}_{p^{n}})\backslash\{(0,0)\}.\]
Note that $|A||T|=p^{n+1}$, by Lemma~\ref{prelemma2}, $(A,T)$ is a tiling pair.

{\bf{Case 2: there exists $s\in[0,n-1]$ such that $(0,p^{s})\in\mathcal{Z}_{A}$.}}

Define
\[T=\{(x,\sum_{i\in[0,n-1]\backslash\{n-s-1\}}y_{i}p^{i}:\ x,y_{i}\in[0,p-1]\}.\]
For any $j\in[0,p-1]$ and $t\in[0,p^{n-1}-1]$, if $(x,\sum_{i\in[0,n-1]\backslash\{n-s-1\}}y_{i}p^{i})\in H_{T}((1,0),t+jp^{n-1})$, then
\begin{align*}
&x\equiv j\pmod{p},\\
&t=0.
\end{align*}
 Hence
\[|H_{T}((1,0),t+jp^{n-1})|=\begin{cases}p^{n-1},&\text{ if }t=0;\\
0,&\text{ if }t\ne0,\end{cases}\]
for any $j\in[0,p-1]$ and $t\in[0,p^{n-1}-1]$.
By Lemma~\ref{lemma1},
 \begin{align}
 (1,0)\in\mathcal{Z}_{T}.\label{eq4}
 \end{align}
 For any $k\in[0,n-1]\backslash\{s\}$, $c\in\mathbb{Z}_{p}$, $t\in[0,p^{n-1-k}-1]$ and $j\in[0,p-1]$, if $(x,\sum_{i\in[0,n-1]\backslash\{n-s-1\}}y_{i}p^{i})\in H_{T}((c,p^{k}),tp^{k}+jp^{n-1})$, then
 \[cxp^{n-1}+(\sum_{i\in[0,n-1]\backslash\{n-s-1\}}y_{i}p^{i})p^{k}\equiv tp^{k}+jp^{n-1}\pmod{p^{n}},\]
 which leads to
 \begin{align*}
 &cx+y_{n-1-k}\equiv j\pmod{p},\\
 &\sum_{i\in[0,n-1]\backslash\{n-s-1\}}y_{i}p^{i}\equiv t\pmod{p^{n-1-k}}.
\end{align*}
Then we have
\[|H_{T}((c,p^{k}),tp^{k}+jp^{n-1})|=\begin{cases}p^{k+1},&\text{ if }k<s\text{ and }t[n-s-1]=0;\\
0,&\text{ if }k<s\text{ and }t[n-s-1]\ne0;\\
p^{k},&\text{ if }k>s,\end{cases}\]
for $k\in[0,n-1]\backslash\{s\}$, $c\in\mathbb{Z}_{p}$, $t\in[0,p^{n-1-k}-1]$ and $j\in[0,p-1]$.
By Lemma~\ref{lemma1},
\begin{align}
(c,p^{k})\in\mathcal{Z}_{T}\text{ for all }k\in[0,n-1]\backslash\{s\}, c\in\mathbb{Z}_{p}. \label{eq5}
\end{align}
For any $c\in\mathbb{Z}_{p}^{*}$, $t\in[0,p^{n-1-s}-1]$ and $j\in[0,p-1]$, if $(x,\sum_{i\in[0,n-1]\backslash\{n-s-1\}}y_{i}p^{i})\in H_{T}((c,p^{s}),tp^{s}+jp^{n-1})$, then
 \[cxp^{n-1}+(\sum_{i\in[0,n-1]\backslash\{n-s-1\}}y_{i}p^{i})p^{s}\equiv tp^{s}+jp^{n-1}\pmod{p^{n}},\]
 which leads to
 \begin{align*}
 &\sum_{i\in[0,n-s-2]}y_{i}p^{i}\equiv t\pmod{p^{n-s-1}},\\
 &cx\equiv j\pmod{p}.
 \end{align*}
Then we have
\[|H_{T}((c,p^{s}),tp^{s}+jp^{n-1})|=p^{s}\]
for all $c\in\mathbb{Z}_{p}^{*}$, $t\in[0,p^{n-1-s}-1]$ and $j\in[0,p-1]$.
By Lemma~\ref{lemma1},
\begin{align}
(c,p^{s})\in\mathcal{Z}_{T}\text{ for all }c\in\mathbb{Z}_{p}^{*}.\label{eq6}
\end{align}
Combining (\ref{eq4}), (\ref{eq5}) and (\ref{eq6}), we have
\[\mathcal{Z}_{A}\cup\mathcal{Z}_{T}=(\mathbb{Z}_{p}\times\mathbb{Z}_{p^{n}})\backslash\{(0,0)\}.\]
Note that $|A||T|=p^{n+1}$, by Lemma~\ref{prelemma2}, $(A,T)$ is a tiling pair.

{\bf{Case 3: there exist $c\in\mathbb{Z}_{p}^{*}$ and $s\in[0,n-1]$ such that $(c,p^{s})\in\mathcal{Z}_{A}$}.}

Define
\[T=\{(-c^{-1}y_{n-s-1},\sum_{i=0}^{n-1}y_{i}p^{i}):\ y_{i}\in[0,p-1]\}.\]
For any $t\in[0,p^{n-1}-1]$ and $j\in[0,p-1]$, if $(-c^{-1}y_{n-s-1},\sum_{i=0}^{n-1}y_{i}p^{i})\in H_{T}((1,0),t+jp^{n-1})$, then
\begin{align*}
&-c^{-1}y_{n-s-1}\equiv j\pmod{p},\\
&t=0.
\end{align*}
Hence
\[|H_{T}((1,0),t+jp^{n-1})|=\begin{cases}p^{n-1},&\text{ if }t=0;\\
0,&\text{ if }t\ne0,\end{cases}\]
for any $t\in[0,p^{n-1}-1]$ and $j\in[0,p-1]$.
By Lemma~\ref{lemma1},
 \begin{align}
 (1,0)\in\mathcal{Z}_{T}.\label{eq7}
 \end{align}
 For any $k\in[0,n-1]\backslash\{s\}$, $d\in\mathbb{Z}_{p}$, $t\in[0,p^{n-1-k}-1]$ and $j\in[0,p-1]$, if $(-c^{-1}y_{n-s-1},\sum_{i=0}^{n-1}y_{i}p^{i})\in H_{T}((d,p^{k}),tp^{k}+jp^{n-1})$, then
 \[-dc^{-1}y_{n-s-1}p^{n-1}+(\sum_{i=0}^{n-1}y_{i}p^{i})p^{k}\equiv tp^{k}+jp^{n-1}\pmod{p^{n}},\]
 which leads to
 \begin{align*}
 &-dc^{-1}y_{n-s-1}+y_{n-1-k}\equiv j\pmod{p},\\
 &\sum_{i=0}^{n-2-k}y_{i}p^{i}\equiv t\pmod{p^{n-1-k}}.
\end{align*}
Then we have
\[|H_{T}((d,p^{k}),tp^{k}+jp^{n-1})|=p^{k}\]
for $k\in[0,n-1]\backslash\{s\}$, $d\in\mathbb{Z}_{p}$, $t\in[0,p^{n-1-k}-1]$ and $j\in[0,p-1]$.
By Lemma~\ref{lemma1},
\begin{align}
(d,p^{k})\in\mathcal{Z}_{T}\text{ for all }k\in[0,n-1]\backslash\{s\}, d\in\mathbb{Z}_{p}. \label{eq8}
\end{align}
For any $d\in\mathbb{Z}_{p}$, $d\ne c$, $t\in[0,p^{n-1-s}-1]$ and $j\in[0,p-1]$, if $(-c^{-1}y_{n-s-1},\sum_{i=0}^{n-1}y_{i}p^{i})\in H_{T}((d,p^{s}),tp^{s}+jp^{n-1})$, then
 \[-dc^{-1}y_{n-s-1}p^{n-1}+(\sum_{i=0}^{n-1}y_{i}p^{i})p^{s}\equiv tp^{s}+jp^{n-1}\pmod{p^{n}},\]
 which leads to
 \begin{align*}
 &\sum_{i=0}^{n-s-2}y_{i}p^{i}\equiv t\pmod{p^{n-s-1}},\\
 &-dc^{-1}y_{n-s-1}+y_{n-s-1}\equiv(1-dc^{-1})y_{n-s-1}\equiv j\pmod{p}.
 \end{align*}
Then we have
\[|H_{T}((d,p^{s}),tp^{s}+jp^{n-1})|=p^{s},\]
for all $d\in\mathbb{Z}_{p}$ with $d\ne c$, $t\in[0,p^{n-1-s}-1]$ and $j\in[0,p-1]$.
By Lemma~\ref{lemma1},
\begin{align}
(d,p^{s})\in\mathcal{Z}_{T}\text{ for all }d\in\mathbb{Z}_{p},d\ne c.\label{eq9}
\end{align}
Combining (\ref{eq7}), (\ref{eq8}) and (\ref{eq9}), we have
\[\mathcal{Z}_{A}\cup\mathcal{Z}_{T}=(\mathbb{Z}_{p}\times\mathbb{Z}_{p^{n}})\backslash\{(0,0)\}.\]
Note that $|A||T|=p^{n+1}$, by Lemma~\ref{prelemma2}, $(A,T)$ is a tiling pair.
\end{proof}

\begin{theorem}
Let $A\subseteq\mathbb{Z}_{p}\times\mathbb{Z}_{p^{n}}$ be a spectral set. If $|A|=p^{s}$ with $2\le s\le n$, then $A$ is a tiling set.
\end{theorem}
\begin{proof}
Suppose $(A,B)$ is a spectral pair, then $|A|=|B|=p^{s}$. Let $I=\{i\in[0,n-1]: (0,p^{i})\in\mathcal{Z}_{A}\}$ and $J=\{i\in[0,n-1]: (0,p^{i})\in\mathcal{Z}_{B}\}$.
Since $(B-B)\backslash\{(0,0)\}\subseteq\mathcal{Z}_{A}$ and $(A-A)\backslash\{(0,0)\}\subseteq\mathcal{Z}_{B}$, then $|I|\ge s-1$ and $|J|\ge s-1.$
By Lemma~\ref{lemma2}, we have $s-1\le|I|\le s$ and $s-1\le|J|\le s.$

{\bf{Case 1: $|I|=s$.}}

Define
\[T=\{(x,\sum_{i\in [0,n-1]\backslash I}y_{i}p^{n-1-i}): x,y_{i}\in[0,p-1]\}.\]
For any $t\in[0,p^{n-1}-1]$ and $j\in[0,p-1]$, if $(x,\sum_{i\in [0,n-1]\backslash I}y_{i}p^{n-1-i})\in H_{T}((1,0),t+jp^{n-1})$, then
\begin{align*}
&x\equiv j\pmod{p},\\
&t=0.
\end{align*}
Hence
\[|H_{T}((1,0),t+jp^{n-1})|=\begin{cases}p^{n-s},&\text{ if }t=0;\\
0,&\text{ if }t\ne0,\end{cases}\]
for any $j\in[0,p-1]$ and $t\in[0,p^{n-1}-1]$.
By Lemma~\ref{lemma1},
 \begin{align}
 (1,0)\in\mathcal{Z}_{T}.\label{eq10}
 \end{align}
 For any $k\in I$, $c\in\mathbb{Z}_{p}^{*}$, $t\in[0,p^{n-1-k}-1]$ and $j\in[0,p-1]$, if $(x,\sum_{i\in [0,n-1]\backslash I}y_{i}p^{n-1-i})\in H_{T}((c,p^{k}),tp^{k}+jp^{n-1})$, then
 \[xcp^{n-1}+(\sum_{i\in [0,n-1]\backslash I}y_{i}p^{n-1-i})p^{k}\equiv tp^{k}+jp^{n-1}\pmod{p^{n}},\]
 which leads to
 \begin{align*}
 &xc\equiv j\pmod{p},\\
 &\sum_{i\in [0,n-1]\backslash I}y_{i}p^{n-1-i}\equiv t\pmod{p^{n-1-k}}.
\end{align*}
Then we have
\[|H_{T}((c,p^{k}),tp^{k}+jp^{n-1})|=\begin{cases}p^{t_{0}},&\text{ if }t[i]=0\text{ for }i<n-1-k, n-1-i\in I;\\
0,&\text{ otherwise},\end{cases}\]
for $k\in I$, $c\in\mathbb{Z}_{p}^{*}$, $t\in[0,p^{n-1-k}-1]$ and $j\in[0,p-1]$, where $t_{0}=|\{i\in [0,n-1]\backslash I: i<k\}|$.
By Lemma~\ref{lemma1},
\begin{align}
(c,p^{k})\in\mathcal{Z}_{T}\text{ for all }k\in I,\ c\in\mathbb{Z}_{p}^{*}. \label{eq11}
\end{align}
For any $k\in [0,n-1]\backslash I$, $c\in\mathbb{Z}_{p}$, $t\in[0,p^{n-1-k}-1]$ and $j\in[0,p-1]$, if $(x,\sum_{i\in [0,n-1]\backslash I}y_{i}p^{n-1-i})\in H_{T}((c,p^{k}),tp^{k}+jp^{n-1})$, then
 \[cxp^{n-1}+(\sum_{i\in [0,n-1]\backslash I}y_{i}p^{n-1-i})p^{k}\equiv tp^{k}+jp^{n-1}\pmod{p^{n}},\]
 which leads to
 \begin{align*}
 &\sum_{i\in [0,n-1]\backslash I}y_{i}p^{n-1-i}\equiv t\pmod{p^{n-1-k}},\\
 &cx+y_{k}\equiv j\pmod{p}.
 \end{align*}
Then we have
\[|H_{T}((c,p^{k}),tp^{k}+jp^{n-1})|=\begin{cases}p^{t_{0}},&\text{ if }t[i]=0\text{ for }i<n-1-k, n-1-i\in I;\\
0,&\text{ otherwise},\end{cases}\]
for all $k\in [0,n-1]\backslash I$, $c\in\mathbb{Z}_{p}$, $t\in[0,p^{n-1-k}-1]$ and $j\in[0,p-1]$, where $t_{0}=1+|\{i\in [0,n-1]\backslash I: i<k\}|$.
By Lemma~\ref{lemma1},
\begin{align}
(c,p^{k})\in\mathcal{Z}_{T}\text{ for all }k\in [0,n-1]\backslash I, c\in\mathbb{Z}_{p}.\label{eq12}
\end{align}
Combining (\ref{eq10}), (\ref{eq11}) and (\ref{eq12}), we have
\[\mathcal{Z}_{A}\cup\mathcal{Z}_{T}=(\mathbb{Z}_{p}\times\mathbb{Z}_{p^{n}})\backslash\{(0,0)\}.\]
Note that $|A||T|=p^{n+1}$, by Lemma~\ref{prelemma2}, $(A,T)$ is a tiling pair.

{\bf{Case 2: $|J|=s-1$.}}

Define
\[T=\{(0,\sum_{i\in [0,n-1]\backslash J}y_{i}p^{i}): y_{i}\in[0,p-1]\}.\]
For any $(a_{1},x_{1})\ne(a_{2},x_{2})\in A$ and $(0,\sum_{i\in [0,n-1]\backslash J}y_{i}p^{i})\ne(0,\sum_{i\in [0,n-1]\backslash J}z_{i}p^{i})\in T$, if
\[(a_{1},x_{1})-(a_{2},x_{2})=(0,\sum_{i\in [0,n-1]\backslash J}y_{i}p^{i})-(0,\sum_{i\in [0,n-1]\backslash J}z_{i}p^{i}),\]
then
\begin{align*}
&a_{1}=a_{2},\\
&x_{1}-x_{2}=\sum_{i\in [0,n-1]\backslash J}(y_{i}-z_{i})p^{i}.
\end{align*}
Since $J=\{i\in[0,n-1]: (0,p^{i})\in\mathcal{Z}_{B}\}$ and $(A-A)\backslash\{(0,0)\}\subseteq\mathcal{Z}_{B}$, then $M(x_{1}-x_{2})\in J$, which is contradicting to $M(\sum_{i\in [0,n-1]\backslash J}(y_{i}-z_{i})p^{i})\in[0,n-1]\backslash J$. Hence $(A-A)\cap(T-T)=\{(0,0)\}$.
Note that $|A||T|=p^{n+1}$, by Lemma~\ref{prelemma2}, $(A,T)$ is a tiling pair.

{\bf{Case 3: $|I|=s-1$ and $|J|=s$.}}

Since $|B|=p^{s}$, $(B-B)\backslash\{(0,0)\}\subseteq\mathcal{Z}_{A}$ and $|I|=s-1$. Let
\[B_{t}=\{(x,y)\in B: x=t\},\]
then for any $(t,y),(t,y')\in B_{t}$, we have $M(y-y')\in I$. Then $|B_{t}|\le p^{s-1}$. But $p^{s}=|B|=|\cup_{t=0}^{p-1}B_{t}|\le p^{s}$, hence $|B_{t}|=p^{s-1}$ for all $t\in[0,p-1]$. Thus, for any $i\in I$, there exist $(t,y),(t,y')\in B_{t}$ such that $M(y-y')=i$.
Since $|J|>|I|$, then there exists $j_{0}\in J$ such that $n-1-j_{0}\not\in I$.
Recall that $n-1-j_{0}\not\in I$, then for any $(t,x),(t,x')\in B$, if $x\equiv x'\pmod{p^{n-1-j_{0}}}$, then $x\equiv x'\pmod{p^{n-j_{0}}}$.
Note that $(0,p^{j_{0}})\in\mathcal{Z}_{B}$ and
\[\langle(t,x),(0,p^{j_{0}})\rangle=xp^{j_{0}},\]
 then for any $(0,x_{0})\in B$, there exist $(1,x_{1}),\dots,(p-1,x_{p-1})\in B$ such that
\[\{x_{i}-x_{0}: i\in[0,p-1]\}\equiv\{i\cdot p^{n-1-j_{0}}:\ i\in[0,p-1]\}\pmod{p^{n-j_{0}}}.\]
Therefore, for any $i\in I$ with $i<n-1-j_{0}$, there exist $(0,x),(0,x_{0}),(1,x_{1}),\dots,(p-1,x_{p-1})\in B$ such that
 \begin{align*}
 &M(x-x_{0})=i,\\
 &\{x_{j}-x_{0}: j\in[0,p-1]\}\equiv\{i\cdot p^{n-1-j_{0}}:\ i\in[0,p-1]\}\pmod{p^{n-j_{0}}}.
 \end{align*}
 Then
 $(j,x_{j}-x)\in B-B$ and
 \[\{(j,x_{j}-x): j\in[0,p-1]\}\sim\{(d,p^{i}): d\in\mathbb{Z}_{p}\}.\]
 Hence we have
 \begin{align}
(d,p^{i})\in\mathcal{Z}_{A}\text{ for all }d\in\mathbb{Z}_{p},\ i\in I\text{ with } i<n-1-j_{0}.\label{eq17}
\end{align}

As we have proved, there exist $(t,x),(t',x')\in B$ such that $t\ne t'$ and $x-x'\equiv c'p^{n-1-j_{0}}\pmod{p^{n-j_{0}}}$ for some $c'\in[1,p-1]$. Denote $c:=c'\cdot (t-t')^{-1}$, then $(t,x)-(t',x')\sim(c^{-1},p^{n-1-j_{0}})\in\mathcal{Z}_{A}$. Define
\[T=\{(-cy_{n-1-j_{0}},\sum_{i\in [0,n-1]\backslash I}y_{i}p^{n-1-i}): y_{i}\in[0,p-1]\}.\]
For any $t\in[0,p^{n-1}-1]$ and $j\in[0,p-1]$, if $(-cy_{n-1-j_{0}},\sum_{i\in [0,n-1]\backslash I}y_{i}p^{n-1-i})\in H_{T}((1,0),t+jp^{n-1})$, then
\begin{align*}
&-cy_{n-1-j_{0}}\equiv j\pmod{p},\\
&t=0.
\end{align*}
Hence
\[|H_{T}((1,0),t+jp^{n-1})|=\begin{cases}p^{n-s},&\text{ if }t=0;\\
0,&\text{ if }t\ne0,\end{cases}\]
for any $j\in[0,p-1]$ and $t\in[0,p^{n-1}-1]$.
By Lemma~\ref{lemma1},
 \begin{align}
 (1,0)\in\mathcal{Z}_{T}.\label{eq13}
 \end{align}
 For any $k\in I$ with $k>n-1-j_{0}$, $d\in\mathbb{Z}_{p}^{*}$, $t\in[0,p^{n-1-k}-1]$ and $j\in[0,p-1]$, if $(-cy_{n-1-j_{0}},\sum_{i\in [0,n-1]\backslash I}y_{i}p^{n-1-i})\in H_{T}((d,p^{k}),tp^{k}+jp^{n-1})$, then
 \[-dcy_{n-1-j_{0}}p^{n-1}+(\sum_{i\in [0,n-1]\backslash I}y_{i}p^{n-1-i})p^{k}\equiv tp^{k}+jp^{n-1}\pmod{p^{n}},\]
 which leads to
 \begin{align*}
 &-dcy_{n-1-j_{0}}\equiv j\pmod{p},\\
 &\sum_{i\in [0,n-1]\backslash I}y_{i}p^{n-1-i}\equiv t\pmod{p^{n-1-k}}.
\end{align*}
Then we have
\[|H_{T}((d,p^{k}),tp^{k}+jp^{n-1})|=\begin{cases}p^{t_{0}},&\text{ if }t[i]=0\text{ for }i<n-1-k, n-1-i\in I;\\
0,&\text{ otherwise},\end{cases}\]
for $k\in I$ with $k>n-1-j_{0}$, $d\in\mathbb{Z}_{p}^{*}$, $t\in[0,p^{n-1-k}-1]$ and $j\in[0,p-1]$, where $t_{0}=|\{i\in [0,n-1]\backslash I: i<k\}|-1$.
By Lemma~\ref{lemma1},
\begin{align}
(d,p^{k})\in\mathcal{Z}_{T}\text{ for all }k\in I\text{ with } k>n-1-j_{0},\ d\in\mathbb{Z}_{p}^{*}. \label{eq14}
\end{align}
For any $k\in [0,n-1]\backslash I$ with $k\ne n-1-j_{0}$, $d\in\mathbb{Z}_{p}$, $t\in[0,p^{n-1-k}-1]$ and $j\in[0,p-1]$, if $(-cy_{n-1-j_{0}},\sum_{i\in [0,n-1]\backslash I}y_{i}p^{n-1-i})\in H_{T}((d,p^{k}),tp^{k}+jp^{n-1})$, then
 \[-dcy_{n-1-j_{0}}p^{n-1}+(\sum_{i\in [0,n-1]\backslash I}y_{i}p^{n-1-i})p^{k}\equiv tp^{k}+jp^{n-1}\pmod{p^{n}},\]
 which leads to
 \begin{align*}
 &\sum_{i\in [0,n-1]\backslash I}y_{i}p^{n-1-i}\equiv t\pmod{p^{n-1-k}},\\
 &-dcy_{n-1-j_{0}}+y_{k}\equiv j\pmod{p}.
 \end{align*}
Then we have
\[|H_{T}((d,p^{k}),tp^{k}+jp^{n-1})|=\begin{cases}p^{t_{0}},&\text{ if }t[i]=0\text{ for }i<n-1-k, n-1-i\in I;\\
0,&\text{ otherwise},\end{cases}\]
for all $k\in [0,n-1]\backslash I$ with $k\ne n-1-j_{0}$, $d\in\mathbb{Z}_{p}$, $t\in[0,p^{n-1-k}-1]$ and $j\in[0,p-1]$, where $t_{0}=|\{i\in [0,n-1]\backslash I: i<k\}|$.
By Lemma~\ref{lemma1},
\begin{align}
(d,p^{k})\in\mathcal{Z}_{T}\text{ for all }k\in [0,n-1]\backslash I\text{ with } k\ne n-1-j_{0},\  d\in\mathbb{Z}_{p}.\label{eq15}
\end{align}
For any $d\in\mathbb{Z}_{p}$ with $d\ne c^{-1}$, $t\in[0,p^{j_{0}}-1]$ and $j\in[0,p-1]$, if $(-cy_{n-1-j_{0}},\sum_{i\in [0,n-1]\backslash I}y_{i}p^{n-1-i})\in H_{T}((d,p^{n-1-j_{0}}),tp^{n-1-j_{0}}+jp^{n-1})$, then
 \[-dcy_{n-1-j_{0}}p^{n-1}+(\sum_{i\in [0,n-1]\backslash I}y_{i}p^{n-1-i})p^{n-1-j_{0}}\equiv tp^{n-1-j_{0}}+jp^{n-1}\pmod{p^{n}},\]
 which leads to
 \begin{align*}
 &\sum_{i\in [0,n-1]\backslash I}y_{i}p^{n-1-i}\equiv t\pmod{p^{j_{0}}},\\
 &(1-dc)y_{n-1-j_{0}}\equiv j\pmod{p}.
 \end{align*}
Then we have
\[|H_{T}((d,p^{n-1-j_{0}}),tp^{n-1-j_{0}}+jp^{n-1})|=\begin{cases}p^{t_{0}},&\text{ if }t[i]=0\text{ for }i<j_{0}, n-1-i\in I;\\
0,&\text{ otherwise},\end{cases}\]
for all $d\in\mathbb{Z}_{p}$ with $d\ne c^{-1}$, $t\in[0,p^{j_{0}}-1]$ and $j\in[0,p-1]$, where $t_{0}=|\{i\in [0,n-1]\backslash I: i<n-1-j_{0}\}|$.
By Lemma~\ref{lemma1},
\begin{align}
(d,p^{n-1-j_{0}})\in\mathcal{Z}_{T}\text{ for all } d\in\mathbb{Z}_{p}\text{ with } d\ne c^{-1}.\label{eq16}
\end{align}
Combining (\ref{eq17}), (\ref{eq13}), (\ref{eq14}), (\ref{eq15}) and (\ref{eq16}), we have
\[\mathcal{Z}_{A}\cup\mathcal{Z}_{T}=(\mathbb{Z}_{p}\times\mathbb{Z}_{p^{n}})\backslash\{(0,0)\}.\]
Note that $|A||T|=p^{n+1}$, by Lemma~\ref{prelemma2}, $(A,T)$ is a tiling pair.
\end{proof}

\begin{theorem}
Let $A\subseteq\mathbb{Z}_{p}\times\mathbb{Z}_{p^{n}}$. If $|A|=mp^{s}$ with $2\le m\le p-1$, $1\le s\le n-1$, then $A$ is not a spectral set.
\end{theorem}
\begin{proof}
Suppose $(A,B)$ is a spectral pair, then $|A|=|B|=mp^{s}$.
Let $I=\{i\in[0,n-1]: (0,p^{i})\in\mathcal{Z}_{A}\}$ and $J=\{i\in[0,n-1]: (0,p^{i})\in\mathcal{Z}_{B}\}$.
Since $(B-B)\backslash\{(0,0)\}\subseteq\mathcal{Z}_{A}$ and $(A-A)\backslash\{(0,0)\}\subseteq\mathcal{Z}_{B}$, then $|I|=|J|=s$. Now we divide our discussion into two cases.

{\bf{Case 1: There exists $r\in[0,n-1]$ such that $n-1-r\not\in I$ and $r\not\in J$.}}

We prove this case by induction. It has been proven that Fuglede's conjecture holds in $\mathbb{Z}_{p}\times\mathbb{Z}_{p}$ and $\mathbb{Z}_{p}\times\mathbb{Z}_{p^{2}}$ \cite{IMP17,S20}. We assume that Fuglede's conjecture holds in $\mathbb{Z}_{p}\times\mathbb{Z}_{p^{i}}$ for all $i\le n-1$, now we consider the Fuglede's conjecture in $\mathbb{Z}_{p}\times\mathbb{Z}_{p^{n}}$.

Let $\phi_{1},\phi_{2}:\mathbb{Z}_{p^{n}}\mapsto\mathbb{Z}_{p^{n-1}}$ be defined by
\begin{align*}
&\phi_{1}(\sum_{i=0}^{n-1}c_{i}p^{i})=\sum_{i=0}^{r-1}c_{i}p^{i}+\sum_{i=r}^{n-2}c_{i+1}p^{i},\\
&\phi_{2}(\sum_{i=0}^{n-1}c_{i}p^{i})=\sum_{i=0}^{n-2-r}c_{i}p^{i}+\sum_{i=n-1-r}^{n-2}c_{i+1}p^{i}.
\end{align*}
Let
\[A'=\{(a,\phi_{1}(x)): (a,x)\in A\},\]
and
\[B'=\{(b,\phi_{2}(y)): (b,y)\in B\}.\]
Since $(0,p^{n-1-r})\not\in\mathcal{Z}_{A}$, $(0,p^{r})\not\in\mathcal{Z}_{B}$ and $(B-B)\backslash\{(0,0)\}\subseteq\mathcal{Z}_{A}$, $(A-A)\backslash\{(0,0)\}\subseteq\mathcal{Z}_{B}$, then $|A'|=|A|$ and $|B'|=|B|$. In the following, we will show that $(A',B')$ forms a spectral pair in $\mathbb{Z}_{p}\times\mathbb{Z}_{p^{n-1}}$.

Let
\[i'=\begin{cases}i,&\text{ if }i<n-1-r;\\
i-1,&\text{ if }i>n-1-r.\end{cases}\]
Note that for any $(b_{1},y),(b_{2},w)\in B$, we hvae
\[(b_{1},\phi_{2}(y))-(b_{2},\phi_{2}(w))\sim\begin{cases}(1,0),&\text{ if }(b_{1}-b_{2},y-w)\sim(1,0);\\
(0,p^{i'}),&\text{ if }(b_{1}-b_{2},y-w)\sim(0,p^{i});\\
(c,p^{i'}),&\text{ if }(b_{1}-b_{2},y-w)\sim(c,p^{i})\text{ and }i\ne n-1-r;\\
(d,p^{j}),&\text{ if }(b_{1}-b_{2},y-w)\sim(c,p^{n-1-r}),\ y-w\equiv c_{1}p^{n-1-r}\pmod{p^{n-r}}\\
&\text{ and }M(y-w-c_{1}p^{n-1-r})=j+1;\\
(1,0),&\text{ if }(b_{1}-b_{2},y-w)\sim(c,p^{n-1-r})\text{ and }y-w= c_{1}p^{n-1-r},\end{cases}\]
for some $c,c_{1},d\in\mathbb{Z}_{p}$.

If $(1,0)\in\mathcal{Z}_{A}$, then for any $t\in[0,p-1]$, we have $|\{(a,x)\in A: a=t\}|=\frac{|A|}{p}$, hence $|\{(a,\phi_{1}(x))\in A': a=t\}|=\frac{|A'|}{p}$. Therefore,
\begin{align}
(1,0)\in\mathcal{Z}_{A'}\text{ if }(1,0)\in\mathcal{Z}_{A}.\label{steq1}
\end{align}
If $i\in I$ and $(a,x)\in H_{A}((0,p^{i}),up^{i}+jp^{n-1})$, then $xp^{i}\equiv up^{i}+jp^{n-1}\pmod{p^{n}}$. It follows that $\phi_{1}(x)p^{i'}\equiv \phi_{1}(u)p^{i'}+jp^{n-2}\pmod{p^{n-1}}$, hence
\begin{align}
(0,p^{i'})\in\mathcal{Z}_{A'}\text{ if }(0,p^{i})\in\mathcal{Z}_{A}.\label{steq2}
\end{align}
If $i\ne n-1-r$ and $(a,x)\in H_{A}((c,p^{i}),up^{i}+jp^{n-1})$, then $acp^{n-1}+xp^{i}\equiv up^{i}+jp^{n-1}\pmod{p^{n}}$. It follows that $acp^{n-2}+\phi_{1}(x)p^{i'}\equiv \phi_{1}(u)p^{i'}+jp^{n-2}\pmod{p^{n-1}}$, hence
\begin{align}
(c,p^{i'})\in\mathcal{Z}_{A'}\text{ if }(c,p^{i})\in\mathcal{Z}_{A},\ i\ne n-1-r.\label{steq3}
\end{align}
If $(c,p^{n-1-r})\in\mathcal{Z}_{A}$ for some $c\in\mathbb{Z}_{p}^{*}$, note that
\[\langle(a,x),(c,p^{n-1-r})\rangle_{n}=acp^{n-1}+xp^{n-1-r}\]
and for any $(a,x_{1}),(a,x_{2})\in A$, we have $M(x_{1}-x_{2})\ne r$, then for any $(a,x_{a})\in A$, there exist $(i,x_{i})$, $i\in[0,p-1]\backslash\{a\}$ such that
 \[x_{i}\equiv x_{a}\pmod{p^{r}}\text{ for }i\in[0,p-1].\]
Then $A$ can be represented as
\[A=\cup_{t=0}^{mp^{s-1}-1}A_{t},\]
where $A_{t}=\{(a,x_{a,t}): a\in[0,p-1]\}$ and for any $(a,x_{a,t}),(b,x_{b,t})\in A_{t}$, $x_{a,t}\equiv x_{b,t}\pmod{p^{r}}$. For any $t\in[0,mp^{s-1}-1]$, we can compute to get that
\begin{align*}
&\chi_{(1,0)}(A_{t})=0,\\
&\chi_{(c,p^{i})}(A_{t})=0\text{ for all }c\in\mathbb{Z}_{p}^{*}\text{ and }i\in[n-r,n-1].
\end{align*}
It follows that
\begin{align*}
&\chi_{(1,0)}(A)=0,\\
&\chi_{(d,p^{i})}(A)=0\text{ for all }d\in\mathbb{Z}_{p}^{*}\text{ and }i\in[n-r,n-1].
\end{align*}
 Hence
\begin{align}
(1,0),(d,p^{i})\in\mathcal{Z}_{A}\text{ for all }d\in\mathbb{Z}_{p}^{*}\text{ and }i\in[n-r,n-1],\text{ if }(c,p^{n-1-r})\in\mathcal{Z}_{A}\text{ for some }c\in\mathbb{Z}_{p}^{*}.\label{steq4}
\end{align}
Combining (\ref{steq1}), (\ref{steq2}), (\ref{steq3}) and (\ref{steq4}), we have $(A',B')$ forms a spectral pair in $\mathbb{Z}_{p}\times\mathbb{Z}_{p^{n-1}}$, which is a contradiction.

{\bf{Case 2: There does not exist $r\in[0,n-1]$ such that $n-1-r\not\in I$ and $r\not\in J$.}}

{\bf{Claim:}} there is at most one $i$ such that $i\in I$ and $n-1-i\not\in J$. Similarly, there is at most one $j$ such that $j\in J$ and $n-1-j\not\in I$.

Assume that there exist $i_{1},i_{2}\in I$ and $n-1-i_{1},n-1-i_{2}\not\in J$.
Since $n-1-i_{1}\not\in J$ and $(A-A)\backslash\{(0,0)\}\subseteq\mathcal{Z}_{B}$, then for any $(t,x),(t,x')\in A$, if $x\equiv x'\pmod{p^{n-1-i_{1}}}$, then $x\equiv x'\pmod{p^{n-i_{1}}}$.
Note that $(0,p^{i_{1}})\in\mathcal{Z}_{A}$ and
\[\langle(t,x),(0,p^{i_{1}})\rangle=xp^{i_{1}},\]
 then for any $(0,x_{0})\in A$, there exist $(1,x_{1}),\dots,(p-1,x_{p-1})\in A$ such that
\[\{x_{i}-x_{0}: i\in[0,p-1]\}\equiv\{i\cdot p^{n-1-i_{1}}: i\in[0,p-1]\}\pmod{p^{n-i_{1}}}.\]
Similarly, there exist $(1,y_{1}),\dots,(p-1,y_{p-1})\in A$ such that
\[\{y_{i}-x_{0}: i\in[0,p-1]\}\equiv\{i\cdot p^{n-1-i_{2}}: i\in[0,p-1]\}\pmod{p^{n-i_{2}}}.\]
If $i_{1}>i_{2}$, then
\[(0,y_{1}-x_{1})\sim(0,p^{n-1-i_{1}})\in \mathcal{Z}_{B},\]
 which is a contradiction. Similarly, we can get a contradiction for the case $i_{1}<i_{2}$. This ends the proof of the claim.

By the claim, we have $s=n-1$, i.e., $|A|=|B|=mp^{n-1}$. We may assume $I=[0,n-1]\backslash\{i_{0}\}$ and $J=[0,n-1]\backslash\{j_{0}\}$, then $n-1-i_{0}\in J$ and $n-1-j_{0}\in I$.

Since $i_{0}\not\in I$, then for any $(b,y),(b,y')\in B$, if $y\equiv y'\pmod{p^{i_{0}}}$, then $y\equiv y'\pmod{p^{i_{0}+1}}$.
Note that $(0,p^{n-1-i_{0}})\in\mathcal{Z}_{B}$ and
\[\langle(b,y),(0,p^{n-1-i_{0}})\rangle=yp^{n-1-i_{0}},\]
then for any $(0,y_{0})\in B$, there exist $(1,y_{1}),\dots,(p-1,y_{p-1})\in B$ such that
\[\{y_{i}-y_{0}: i\in[0,p-1]\}\equiv\{i\cdot p^{i_{0}}: i\in[0,p-1]\}\pmod{p^{i_{0}+1}}.\]
Hence
\begin{align}
(1,y_{1}-y_{0})\sim(c_{0},p^{i_{0}})\in\mathcal{Z}_{A}\text{ for some }c_{0}\in\mathbb{Z}_{p}^{*}.\label{steq5}
\end{align}
Since $|B|=mp^{n-1}$, then for any $k<i_{0}$, there exist $(t,y),(t,y_{t})\in B$ such that $M(y_{t}-y)=k$,
and there exist $(j,y_{j})\in B$, $j\in[0,p-1]\backslash\{t\}$ such that
\[\{y_{j}-y_{t}: j\in[0,p-1]\}\equiv\{i\cdot p^{i_{0}}: i\in[0,p-1]\}\pmod{p^{i_{0}+1}}.\]
 Then
 $(j-t,y_{j}-y)\in B-B$ and
 \[\{(j-t,y_{j}-y): j\in[0,p-1]\}\sim\{(c,p^{k}): c\in\mathbb{Z}_{p}\}.\]
Hence we have
 \begin{align}
(c,p^{k})\in\mathcal{Z}_{A}\text{ for all }c\in\mathbb{Z}_{p},\ k<i_{0}.\label{steq6}
\end{align}
Define
\[D=\{\sum_{i\in I}a_{i}(0,p^{i})+a_{i_{0}}(c_{0},p^{i_{0}}): a_{i}\in[0,p-1]\}.\]
For any $d\ne d'\in D$, $d-d'=\sum_{i\in I}d_{i}(0,p^{i})+d_{i_{0}}(c_{0},p^{i_{0}})$, where $d_{i}\in[-p+1,p-1]$.
Note that
 \[\sum_{i\in I}d_{i}(0,p^{i})+d_{i_{0}}(c_{0},p^{i_{0}})\sim\begin{cases}(c_{0},p^{i_{0}}),&\text{ if }d_{i_0}\ne0\text{ and } d_{i}=0\text{ for all }i< i_{0};\\
(d_{k}^{-1}d_{i_0}c_{0},p^{k}),&\text{ if }d_{i_0}\ne0,\ d_{k}\ne0,\ k<i_{0} \text{ and }d_{i}=0\text{ for all }i<k;\\
(0,p^{k}),&\textup{ if }d_{i_0}=0,\ d_{k}\ne0\text{ and }d_{i}=0\text{ for all }i<k.\end{cases}\]
By (\ref{steq5}) and (\ref{steq6}), we have $(D-D)\backslash\{(0,0)\}\subseteq\mathcal{Z}_{A}$, then $D$ is a spectrum for $A$, but $|D|=p^{n}>|A|$, which is a contradiction. This completes the proof.
\end{proof}

\section{Conclusion}\label{conclusion}
From the proof of Theorem~\ref{mainthm}, we can see that the divisibility property and equi-distributed property in $\mathbb{Z}_{p}\times\mathbb{Z}_{p^{n}}$ play an important role. In \cite{S20}, the author had proved that the equi-distributed property holds in $\mathbb{Z}_{p^{m}}\times\mathbb{Z}_{p^{n}}$. Actually, the divisibility property can also be generalized to the group $\mathbb{Z}_{p^{m}}\times\mathbb{Z}_{p^{n}}$.
\begin{lemma}\label{lemma5}
Let $m,n$ be integers with $m\le n$. Let $A\subseteq\mathbb{Z}_{p^{m}}\times\mathbb{Z}_{p^{n}}$. If $(a,p^{i_{1}}),\ (0,p^{i_{2}}),\\ \dots,(0,p^{i_{s}})\in\mathcal{Z}_{A}$ for some $a\in p^{m-1}\mathbb{Z}_{p^{m}}$ and $0\le i_{1}<i_{2}<\dots<i_{s}\le n-1$, then $p^{s}\mid|A|$.
\end{lemma}
The proof of Lemma~\ref{lemma5} is similar with the proof of Lemma~\ref{lemma2}, we omit the details. In order to study the zero set in $\mathbb{Z}_{p^{m}}\times\mathbb{Z}_{p^{n}}$, we define an equivalence relation. For $u,v\in\mathbb{Z}_{p^{m}}\times\mathbb{Z}_{p^{n}}$,
we define the relation $u\sim v$ if there exists $r\in\mathbb{Z}_{p^{n}}^{*}$ such that $u=rv$. Then the equivalent classes in $\mathbb{Z}_{p^{m}}\times\mathbb{Z}_{p^{n}}$ by $\sim$ are
\[(p^{i},0),(c_{i}p^{i},p^{j})\text{ for all $i\in[0,m-1]$, $j\in[0,n-1]$ and $c_{i}\in[0,p^{m-i}-1]$ with $\gcd(c_{i},p)=1$.}\]
Note that the first zero in Lemma~\ref{lemma5} is $(a,p^{i_{1}})$, where $a\in p^{m-1}\mathbb{Z}_{p^{m}}$, and there are much more equivalent classes in $\mathbb{Z}_{p^{m}}\times\mathbb{Z}_{p^{n}}$, it seems that our method does not work for the general case $\mathbb{Z}_{p^{m}}\times\mathbb{Z}_{p^{n}}$. It would be interesting to consider the Fuglede's conjecture in $\mathbb{Z}_{p^{m}}\times\mathbb{Z}_{p^{n}}$.

\end{document}